\documentclass[reqno,draft,intlimits]{amsart}

\usepackage[active]{srcltx}

\usepackage[reqno]{amsmath}
\usepackage{amssymb}

\usepackage{wrapfig}

\makeatletter
\long\def\@makecaption#1#2{%
  \vskip3pt
  \sbox\@tempboxa{\small#1. #2}%
    \ifdim \wd\@tempboxa >\hsize
    \small#1. #2\par
  \else
    \global \@minipagefalse
    \hb@xt@\hsize{\hfil\box\@tempboxa\hfil}%
  \fi
  \vskip0pt}
\makeatother

\newtheorem{proc}{Proposition}[section] 
       
\newtheorem{cor}{Corollary}[section]   
          
\newtheorem{lem}{Lemma}[section]

\newtheorem{defi}{Definition}[section]
\newtheorem{rmk}{Remark}[section]

\def\vk{\varkappa}

 \DeclareMathOperator{\id}{id} 
 \DeclareMathOperator{\Ker}{Ker} 
  
 \DeclareMathOperator{\im}{Im}

\DeclareMathOperator{\Hom}{Hom}  
 \DeclareMathOperator{\Sym}{Sym} 
\DeclareMathOperator{\End}{End}

\def\rw{\rightarrow}

\newcommand{\p}{\partial}

\def\cC{\mathcal C}
\def\cE{\mathcal E}
\newcommand{\cF}{\mathcal F}

\newcommand{\cD}{\mathcal D}
\newcommand{\cA}{\mathcal A}

\newcommand{\R}{\mathbb R}
\newcommand{\Z}{\mathbb Z}
\newcommand{\C}{\mathbb C}

\def\be{\begin{equation}}
\def\ee{\end{equation}}
\def\bea{\begin{eqnarray}}
\def\eea{\end{eqnarray}}

\DeclareFontFamily{OT1}{wncyi}{} \DeclareFontShape{OT1}{wncyi}{m}{it}{
   <5> <6> <7> <8> <9> gen * wncyi
   <10> <10.95> <12> <14.4> <17.28> <20.74> <24.88> wncyi10
  }{}
\DeclareSymbolFont{cyrletters}{OT1}{wncyi}{m}{it} 
\DeclareSymbolFontAlphabet{\cyrmath}{cyrletters} 
\DeclareMathSymbol{\rE}{\cyrmath}{cyrletters}{003} 
\DeclareMathSymbol{\rD}{\cyrmath}{cyrletters}{068} 
\DeclareMathSymbol{\rG}{\cyrmath}{cyrletters}{017} 
\DeclareMathSymbol{\rI}{\cyrmath}{cyrletters}{073} 
\DeclareMathSymbol{\rL}{\cyrmath}{cyrletters}{076} 
\DeclareMathSymbol{\rZ}{\cyrmath}{cyrletters}{090}

\newdimen\theight
\def \refright#1{%
             \vadjust{\setbox0=\hbox{\quad\vtop{\hsize5cm\bf\noindent #1}}%
             \theight=\ht0
             \advance\theight by \dp0    \advance\theight by \lineskip
             \kern -\theight \vbox to \theight{\rightline{\rlap{\box0}}%
             \vss}%
             }}%

\begin{document}

\begin{center}
{\LARGE \bf  Some remarks on contact manifolds, Monge-Amp\`ere equations and solution singularities} 
\end{center}
\vspace{.5cm}
\begin{center}

{\Large A.M.Vinogradov}
\end{center}

\bigskip

\begin{center}
$^{1}$
Levi-Civita Institute, 83050 Santo Stefano del Sole (AV), Italia.
\end{center}
\smallskip

\vspace{1.0cm}

\noindent {\bf Abstract.}
We describe some natural relations connecting contact geometry, classical Monge-Amp\`ere 
equations and theory of singularities of solutions to nonlinear PDEs. They reveal the
hidden meaning of Monge-Amp\`ere  equations and sheds new light on some aspects 
of contact geometry.          

\vspace{5.0cm}
\pagebreak

\tableofcontents

 \section{Introduction}
 In this note we describe some natural relations among subjects in the title. While
 contact geometry and Monge-Amp\`ere equations are classical themes, studied in numerous
 works, their natural ties with the theory of singularities of solutions to nonlinear PDEs, as far 
 as we know, were not yet explicitly established  and, a consequence, duly exploited.  
 
The fact that a (nonlinear) PDE is surrounded by a ``cloud" of  subsidiary equations, which describe  
the behavior of singularities of its generalized solutions is also not well-known. These equations were
introduced by the author in \cite{Vsing} and since that are waiting to be systematically studied. 
A detailed exposition of fundamentals of this theory will be published in \cite{Vin}. It is worth
noticing that some particular subsidiary equations are \emph{implicitly} well-known. For instance, 
such are equations describing wave front propagation in the theory of hyperbolic equations, 
or equations of geometrical optics, which, in fact, describe some kind singularities of solutions 
of Maxwell's equations. We used ``implicitly" to stress that these equations were not
identified as a part of solution singularity theory.

A fundamental problem in this theory is whether it is possible to reconstruct the original 
equation if the subsidiary equations are known. In other words, we are asking whether the laws
describing behavior of singularities of a physical field (continuum media, etc), predetermine
the equations describing this field, (media, etc). This problem will be called the \emph
{reconstruction problem}. In this note we show that, informally speaking, the classical 
Monge-Amp\`ere equations are characterized by the fact that the reconstruction problem 
for them is a tautology. The exact formulation of this assertion requires various facts from 
geometry of jet spaces, including contact geometry, and theory of generalized solutions of 
nonlinear PDEs, which are collected and discussed along the paper. When being put in this perspective they bring into the light some structures and their interrelations whose importance 
was not duly understood even in the case they are \emph{formally known}. For instance, the 
intrinsic definition of a contact structure (see below definition \ref{main}) belongs to this list.
These small things are, in fact, rather useful and enrich even such classical subject as contact
geometry. The fact that the intrinsic definition of contact structures immediately leads to
to the explicit description of contact vector fields (see subsection \ref{cf}) of this point.

For the composition of this note is see the ``contents". Everything (manifolds, vector fields, 
etc) in it is assumed to be smooth. The $C^{\infty}(M)$--module of $k$-th order differential forms
(resp., vector fields) on a manifold $M$ is denoted by $\Lambda^k(M)$ (resp., $D(M)$).

\section{Contact manifolds and generalized solutions}
Let $M$ be an $n$--dimensional manifold. Recall that any projective $C^{\infty}(M)$--module 
$P$ of finite type is canonically identified with the module $\Gamma(\pi)$ of sections of a vector
bundle $\pi$ (see \cite{N}). The fiber $\pi^{-1}(x), \,x\in M,$ of $\pi$ is the quotient module 
$P_x:=P/\mu_x\cdot P$,  $\mu_x=\{f\in C^{\infty}(M)\,\mid\,f(x)=0\},$ considered as an 
$\R$--vector space. The \emph{dimension} of the vector bundle $\pi$ is the \emph{rank} of $P$.

Accordingly, below we treat a regular distribution on a manifold $M$ as a projective submodule 
$\cD$ of the $C^{\infty}(M)$--module $D(M)$ of vector fields on $M$. Since $D(M)_x$ is
canonically identified with the tangent to $M$ at $x$ space $T_xM$, $\cD_x$ may be viewed
as a subspace of $T_xM$ and this way one recovers the standard definition of a distribution
as a family of vector spaces $x\mapsto \cD_x\subset T_xM$.

Put $\vk:=D(M)/\cD$. So, $\vk_x$ is identified with $T_xM/\cD_x$ and this is the reason
to call $\vk$ \emph{normal} to $\cD$ bundle. Obviously, 
$\mathrm{rank}\,\vk=n-\mathrm{rank}\,\cD$.
Recall also that the {\it curvature} of $\cD$ 
is the $C^{\infty}(M)$--bilinear form on $\cD$ with values in $\vk$ defined as
$$
R(X,Y):=[X,Y](\mod\cD), \,X,Y\in\cD.
$$
If the rank of $\vk$ is 1, i.e., the distribution $\cD$ is of codimension 1, then $\cD$ 
will be called \emph{nondegenerate}, if
$$
\cD\ni X\mapsto R(X,\cdot)\in\Hom_{C^{\infty}(M)}(\cD,\vk)
$$
is an isomorphism of $C^{\infty}(M)$--modules. Equivalently, a distribution $\cD$ is nondegenerate 
if its curvature is a nondenerate $\vk$--valued $C^{\infty}(M)$--bilinear form on it.

The following is an intrinsic definition of a contact manifold.
\begin{defi}\label{main}
A \emph{contact manifold} is a manifold $M$ supplied with a nondegenerate distribution 
$\cD$ of codimension 1. Such a distribution $\cD$ is called a \emph{contact} structure on $M$.
\end{defi}

\subsection{Comparison with the standard definition.} 
Recall that the fiber at $x\in M$ of the 
vector bundle associated with the $C^{\infty}(M)$--module $\Hom_{C^{\infty}(M)}(P,Q)$ 
with $P,Q$ being projective $C^{\infty}(M)$--modules of finite type is $\Hom_{\R}(P_x,Q_x)$
(see \cite{N}). 
 In particular, this fiber for $\Hom_{C^{\infty}(M)}(\cD,\vk$ is 
$\Hom_{\R}(\cD_x,\vk_x)$. But $\R$--vector space $\vk_x$ is 1-dimensional. 
So, it is identified with $\R$ by choosing a base vector in it. If $\cD$ is locally given by the Pfaff 
equation $\omega=0, \,\omega\in \Lambda^1(M)$,  $X\in D(M)$ is such that $\omega(X)=1$
and $\nu=X(\mathrm{mod}\,\cD)$, then $\nu_x\in\cD_x$ is a base vector, and 
$\Hom_{C^{\infty}(M)}(\cD,\vk)$ (resp., $\Hom_{\R}(\cD_x,\vk_x)$) is identified
with $\cD^*:=\Hom_{C^{\infty}(M)}(\cD,C^{\infty}(M))$ (resp., $\cD^*_x:=\Hom_{\R}(\cD_x,\R)$).

The following assertion is a direct consequence of the standard formula 
$$
d\omega(X,Y)=X(\omega(Y))-Y(\omega(X))-\omega([X,Y]), \;X,Y\in D(M).
$$
\begin{lem}
With the above identifications the curvature $R$ of $\cD$ is identified with $d\omega\mid_{\cD}$,
i.e., $R(X,Y)=-d\omega(X,Y), \;\forall X,Y\in\cD$.
\end{lem}

This lemma shows that nondegeneracy of $R_x$ is equivalent to nondegeneracy of 
$(d\omega)_x, \,\forall x\in M$, i.e., that the bilinear form on the vector space $\cD_x$
$(d\omega)_x$ is symplectic . This nondegeneracy property may be
also expressed by saying that $d\omega$ is nondegenerate on $\cD$, i.e.,
that 
$$
\cD\ni X\quad\mapsto\quad d\omega(X,\cdot)\mid_{\cD}\in\cD^*
$$
is an isomorphism of $C^{\infty}(M)$--modules, or, equivalently, that $\omega\wedge d\omega^m$
is a local volume form on $M$ assuming that $n=2m+1$. These observations show that definition
\ref{main} is equivalent to the standard one.

Recall that $\omega\in\Lambda(M)$ is a \emph{contact form} on $M$ if $\cD:=\{\omega=0\}$ is
a contact structure on $M$. Contact forms defining the same distribution differ one from another 
by a nowhere vanishing factor $f\in C^{\infty}(M)$. Moreover, contact forms associated with a 
contact distribution are, generally, defined only locally. By these two reasons definition \ref{main} 
is more convenient that the standard one. An illustration of that is in the subsequent subsection.

\subsection{Contact transformatins and vector fields.}\label{cf}
A symmetry of a distribution $\cD$ on $M$ is a diffeomorphism $F\colon M\rw M$ that preserves
$\cD$, i.e., $F(X)\in\cD$ if $X\in\cD$ with $F(X):=(F^*)^{-1}\circ X\circ F^*$ being the image of $X$
via $F$. Traditionally symmetries of a contact structure/manifold are called \emph{contact transformations}.

An \emph{infinitesimal symmetry} of $\cD$ is a vector field $Z\in D(M)$ such that 
$[X,Z]\in\cD, \,\forall X\in\cD$. Alternatively, infinitesimal symmetries of $\cD$ are defined as vector
fields whose flows consist of (local) symmetries of $\cD$. Infinitesimal symmetries of $\cD$ form
a Lie subalgebra of $D(M)$, which we denote by $\Sym\cD$. If $\cD$ is contact, then infinitesimal symmetries of $\cD$ are called \emph{contact vector fields} and some authors use a slightly 
ambiguous $\mathrm{Cont}\,M$ for $\Sym\cD$ if $M$ is a contact manifold (see \cite{KLR}).

Now we shall show that definition \ref{main} directly leads to a natural description of contact
vector fields. To this end, we note that with a vector field $Z\in M$ a homomorphism  $\phi_Z\colon \cD\rw\vk$ of $C^{\infty}(M)$--modules is associated. Namely, if $X\in\cD$,
then $\phi_Z(X)=[X,Z] (\mathrm{mod}\,\cD)$. Since $[fX,Z]=f[X,Z]-Z(f)X$, one sees that $\phi_Z$ 
is a $C^{\infty}(M)$--homomorphism.
\begin{proc}\label{cont}
If $\cD$ is contact, then the $\R$--linear map
$$
\Sym\cD\ni Z \quad\mapsto \quad (Z\;\mathrm{mod}\,\cD)\in\vk
$$
is biunique.
\end{proc}
\begin{proof}
\textbf{Injectivity}. If $Z\in \cD\cap(\Sym\cD)$, then 
$$
R(Z,X)=0, \,\forall X\in\cD \quad\Longleftrightarrow\quad R(Z,\cdot)=0.
$$ 
Hence, by nondegeneracy of $R$, $Z=0$.\\
\textbf{Surjectivity.}
Let $\vk\ni\nu=Z'(\mathrm{mod}\,\cD)$. Since the curvature form $R$ is nondegenerate, 
there is $Y\in\cD$ such that $\phi_{Z'}=R(Y,\cdot)$, i.e.,  
$[X,Z']=[Y, X]\,(\mathrm{mod}\,\cD), \,\forall X\in\cD$. So, $[X,Y+Z']\in\cD, \,\forall X\in\cD$ and, therefore, $Z=Y+Z'$ is contact and 
$Z(\mathrm{mod\,\cD})=Z'(\mathrm{mod}\,\cD)=\nu$.
\end{proof}
If the contact vector field $Z$ corresponds via proposition \ref{cont} to $\nu\in\vk$, then
 $\nu$ is called the \emph{generating function} of $Z$ and $Z$ will be denoted by $X_{\nu}$. 
In other words, if $Z\in D(M)$ is contact, then $Z=X_{\nu}$ with $\nu=Z(\mathrm{mod}\,\cD)$. 
Also, proposition \ref{cont} allows to transfer the Lie algebra structure in $\Sym\cD$ to $\vk$. 
Namely, the transfered  Lie bracket, denoted by $\{\cdot,\cdot\}$, is defined by the relation
$$
X_{\{\mu,\nu\}}=[X_{\mu},X_{\nu}], \quad \mu,\nu\in\vk
$$

The $\R$--linear map $\chi\colon\vk\rw D(M), \,\chi(\nu)=X_{\nu}$, is a 1-st order 
differential operator. To prove it we have to show that 
$$
[f,[g,\chi]](\nu)=0, \; \forall f,g\in C^{\infty}(M), \,\nu\in\vk
$$
(see \cite{N}). In view of proposition \ref{cont} it suffices to show that the vector field
$$
Z=[f,[g,\chi]](\nu)=X_{fg\nu}-fX_{g\nu}-gX_{f\nu}+fgX_{\nu}
$$
is contact and its generating function is trivial. But, obviously, 
$X_{h\nu}-hX_{\nu}\in\cD, \,\forall h\in  C^{\infty}(M)$. This proves that $Z\in\cD$. Next, 
if $X\in \cD$, then
\begin{equation*}
\begin{array}{l}
[X,Z]=[X,X_{fg\nu}]-f[X,X_{g\nu}]-g[X,X_{f\nu}]+\\
\qquad fg[X,X_{\nu}]-X(f)X_{g\nu}-X(g)X_{f\nu}+X(fg)X_{\nu}
\end{array}
\end{equation*}
Each of first 4 terms in this expression, obviously, belongs to $\cD$. The remaining 3 terms
of it may be rewritten in the form
$$
X(f)(gX_{\nu}-X_{g\nu})+X(g)(fX_{\nu}-X_{f\nu})),
$$ 
and, as we have seen before, each of 2 summands  in this expression also belongs to $\cD$. 


\subsection{Jets and generalized solutions of nonlinear PDEs}\label{jets}
Fix a manifold $E^{n+m}$ of dimension $n+m, \,m,n\in\ \Z_+$. The $k$--jet of an $n$--dimensional submanifold $N^n\subset E^{n+m}$ at a point $a\in N$ is the equivalence class of $n$--dimensional
submanifolds of $E$ passing through $a$, which are tangent to $N$ with order $k$. 
It will be denoted by $[N]_a^k$. The totality of such jets  forms a smooth manifold, denoted 
by $J^k(E,n)=\bigcup_{N\subset E, a\in N}[N]_a^k$. There is a natural map
\begin{eqnarray*}
j_k(N):N   {\longrightarrow} & J^k(E,n), \quad a  \longmapsto & [N]_a^k
\end{eqnarray*}
A function $f$ on $J^k(E,n)$ is defined to be \emph{smooth} if for any $n$--dimensional submanifold
$N\subset E$ the function $j_k(N)^*(f)$ is smooth. The so-defined smooth function algebra,
denoted by $\cF_k(E,n)$, supplies $J^k(E,n)$ with a smooth manifold structure. The minimal
distribution on $J^k(E,n)$ such that any submanifold of the form $N_{(k)}=\im j_k(N)\subset J^k(E,n)$ is integral for it is the \emph{$k$--th order contact structure}, or \emph{Cartan distribution} on
$J^k(E,n)$. Denote it by $\cC_k(E,n)$. If  $m=1$, then $\cC_1(E,n)$ is the (classiical) contact 
structure on $J^1(E,n)$. For $k\geq l$ the natural projection 
$$
J^k(E,n) \stackrel{\pi_{k,l}}{\longrightarrow} J^l(E,n),\quad 
[N]_a^k  \longmapsto  [N]_a^l
$$
sends $\cC_k(E,n)$ to $\cC_l(E,n)$.

An integral submanifold $L$ of a distribution $\cD$ is called  \emph{locally maximal} if it
is not contained even locally in an integral submanifold of greater dimension. Submanifolds 
$N_{(k)}$ are locally maximal integral submanifolds of $\cC_k(E,n)$ of dimension $n$. But 
except the case $k=m=1$ (contact geometry) there are locally maximal integral submanifolds 
of other types. More exactly, the \emph{type} of such a submanifold $U$ is the dimension of
$\pi_{k,k-1}(U)$, which may vary from $0$ to $n$. For instance, fibers of projection $\pi_{k,k-1}$ 
are of type $0$.  It should be stressed that the notion of type is \emph{intrinsic}, i.e., can be 
defined exclusively in terms of distribution $\cC_k(E,n)$. In this sense the contact geometry
on $J^1(E^{n+1},n)$ is the only exception to the general case.

Recall that a system of (nonlinear) PDEs of order $k$ is geometrically interpreted as a submanifold
$\cE$ of $J^k(E,n)$ for a suitable $E$. In this approach "usual" solutions of $\cE$ are interpreted 
as submanifolds $N\subset E$ such that $N_{(k)}\subset \cE$. Moreover, this interpretation allows 
one to define an analogue of the notion of generalized solutions in the theory of linear PDEs for nonlinear equations. This is achieved by enlarging the
class of submanifolds of the form $N_{(k)}$ to maximal integral submanifolds of type $n$, called
\emph{R-manifolds}. They play the role of Legendrian submanifolds in contact geometry.

 While an R-manifold  $U$ is, by definition, smooth, its singular point are defined to be singular
 points of the projection $\pi_{k,k-1}\mid_U$. Their totality, denoted by $U_{sing}$, is a union of
 submanifolds with singularities. If $\theta\in U_{sing}$, then the kernel of the differential
 $d_{\theta}(\pi_{k,k-1}\mid_U):T_{\theta}U\rw T_{\pi_{k,k-1}(\theta)}J^{k-1}(E,n)$ at $\theta\in U$
 is not trivial and is called the \emph{bend} of $U$ at $\theta$.  Denote it by 
 $\gimel_{\theta}=\gimel_{\theta,U}$.
We shall use the term $s$-\emph{bend} for a bend of dimension $s$. The notion of a bend is
key in the solution singularities theory (see subsection \ref{sng}).
\begin{defi}
A (generalized) $s$--bend solution of an equation $\cE\subset J^k(E,n)$ is an R-manifold
$U\subset \cE$ such that for any $\theta\in U_{sing}$ the dimension of $\gimel_{\theta,U}$
is $s$.  
\end{defi}  
  
 PDEs differ each other by types of generalized solutions they admit. Hence the problem of
 (local) classification of singularities of $R$--manifolds is a central one in geometrical theory 
of PDEs. In particular, as we shall see below, MA-equations are 
 distinguished by the structure of singularities of their generalized solutions.

\section{Monge-Amp\`ere equations}

\subsection{Classical Monge-Amp\`ere equations} Recall that \emph{classical} Monge-Amp\`ere equations (MAEs) are PDEs in two independent variables of the form
\begin{equation}\label{MA}
N(u_{xx}u_{yy}-u_{xy}^2)+Au_{xx}+Bu_{xy}+Cu_{yy}+D=0
\end{equation}
with $N,A,B,C,D$ being some functions of variables $x,y,u,u_{x},u_{y}$.
In the current literature the term ``Monge-Amp\`ere equation" may refer to one 
of various generalizations of classical MAEs.(see, for instance, \cite{KLR}). 

Definition (\ref{MA}) is descriptive and as such does not reveal the true nature of this class of 
equations. Our main goal here is to discover the hidden meaning behind analytical expression (\ref{MA}). First of all, it is important to stress that equations (\ref{MA}) are, in their turn, (locally) subdivided into 3 classes, namely, elliptic, parabolic, or hyperbolic ones, according to 
$AC-B^2-4ND<0, =0, or >0$, respectively. 

The first step toward revealing this meaning was 
due to S.~Lie, who observed that the class of elliptic (resp., parabolic, or hyperbolic) MAEs  
is invariant with respect to contact transformations. Much later some authors and, first of all, 
Lychagin and his collaborators, interpreted a MA-equation as the condition $\omega\,|_L=0$
 imposed on Legendrian submanifolds $L$'s of a 5--dimensional contact manifold $M$ for a 
 given  \emph{effective} 2-form $\omega$ (see  \cite{Ly}, \cite{Mori}, \cite{KLR}). The same
 condition imposed on Legendrian submanifolds of an arbitrary contact manifold is a natural
 generalization of classical MA-equations to higher dimensions.
Nevertheless, though being coordinate-free, this definition is still of a descriptive character.

\subsection{Jordan algebras of self adjoint operators} 
In this subsection $M$ stands for 
a contact manifold of dimension $2n+1$. The curvature form allows one to distinguish 
an important class of operators in 
$\End_{C^{\infty}(M)}\cD$. Namely, a $C^{\infty}(M)$--homomorphism $A\colon \cD\rw\cD$
is called \emph{self-adjoint} if 
$$
R(AX,Y)=R(X,AY), \,\forall X,Y\in\cD.
$$
\emph{Scalar}, i.e., multiplication by a function operators are, obviously, self-adjoint. 
Self-adjoint operators form a Jordan algebra, denoted by $\mathrm{Sad}(\vk)$, with 
respect to the Jordan product
$$
A*B\colon= \frac{1}{2}(AB+BA).
$$
 $\mathrm{Sad}(\vk)$ is, obviously, a $C^{\infty}(M)$--module, and in what follows  we shall
 concentrate on  \emph{JS-subalgebras} of $\mathrm{Sad}(\vk)$. These are submodules of  
 $\mathrm{Sad}(\vk)$ closed with respect to the Jordan product. Such a subalgebra is 
 \emph{unital} if it contains the identity endomorphism and hence the algebra $C^{\infty}(M)$
 interpreted as a subalgebra of $\End_{C^{\infty}(M)}\cD$. Two JS-algebras are \emph{isomorphic}
 if there is an isomorphism of supporting them $C^{\infty}(M)$--modules preserving the Jordan
 product.
 
 By literally repeating the above definitions in the situation when an $\R$--vector space $V$ 
 and a skew-symmetric bilinear form $\sigma=<\cdot,\cdot>$ on $V$ taking values in a 
 1-dimensional vector space $W$ substitute $M, R$ and $\vk$, respectively, one gets notions
 of ($\sigma$--) self-adjoint operators in $\End\,V$, JS-subalgebras in $\End\,V$, etc.
 
 If $\cA\subset \mathrm{Sad}(\vk)$ is a JS-subalgebra, then its fiber $\cA_x$ at $x\in M$
 inherits a structure of a JS-subalgebra in $\End_{\R}\cD_x$ with respect to the 
 $\vk_x$--valued form $R_x$. By choosing a base vector in $R_x$ one may identify  it with a
 \emph{symplectic} (=skew-symmetric and nondegenerate) bilinear form on the vector space 
 $V=\cD_x$. Recall that a \emph{symplectic vector space} (over $\R$) is an $\R$--vector  
 space supplied with a non-degenerate skew-symmetric form. A unital JS-subalgebra of 
 dimension $n$ with respect to the structure symplectic form on a $2n$--dimensional symplectic 
 vector space will be called \emph{basic}. 
  
 Importance of basic algebras is that they classify geometric solution 
 singularities of (nonlinear) PDEs (see \cite{Vsing}). So, the problem of describing basic
 algebras is fundamental in the solution singularities theory. This problem is not trivial just
 because it includes the problem of describing finite-dimensional commutative algebras over
 $\R$.
 
 By sightly abusing the language we shall call \emph{basic} also a JP-subalgebra $\cA$ of
 $\mathrm{Sad}(\vk)$ such that $\cA_x$ is basic in $\End_{\R}(\cD_x)$ for almost all $x\in M$, i.e.,
 for all $x\in N$ where $N$ is an everywhere dense open in $M$. Almost all fibers $\cA_x$ of 
 $\cA$ are of dimension $n$. By this reason we say that the dimension of $\cA$ is $n$.


\subsection{2-dimensional unital JS-algebras}\label{JS}
Here we shall illustrate the above-said  in the simplest nontrivial case $n=2$. But before we
shall list some elementary properties of symplectic self-adjoint operators for arbitrary $n$.
\begin{lem}\label{elementary}
Let $A$ be an symplectic self-adjont operator on a $2n$--dimensional symplectic vector space
$V$. Then
\begin{enumerate}
\item The symplectic form $<\cdot,\cdot>$ vanishes on any cyclic subspace 
$$
C_v\colon=\mathrm{Span}\{A^k(v)\}_{k\geq 0}, \,v\in V.
$$.
\item Root subspaces of $A$ corresponding to different eigenvalues of $A$ are symplectic
orthogonal.
\item $<\Ker A,\im A>=0$.
\end{enumerate}
\end{lem}
\begin{proof}
(1) This is an obvious consequence of $<Aw,w>=0$ for any $w\in V$. But
$<Aw,w>=<w,Aw>=-<Aw,w>$.

(2) The root subspace of $A$ corresponding to a real (resp., complex) eigenvalue $\lambda$
of it is of the form $\Ker f(A)$ where $f(t)=(t-\lambda)^k$ 
(resp., $f(t)=(t^2-(\lambda+\bar{\lambda})t+\lambda\bar{\lambda})^m)$. If eigenvalues $\lambda_1$ 
and $\lambda_2$ are different, then the corresponding to them polynomials $f_(t)$ and $f_2(t)$
are relatively prime and hence there are polynomials $g_1(t)$ and $g_2(t)$ such that
$$
f_1(t)g_1(t)+f_2(t)g_2(t)=1 \quad \Longrightarrow\quad f_1(A)g_1(A)+f_2(A)g_2(A)=\id_V.
$$
I follows from the last relation that $f_1(A)$ is invertible on $\Ker f_2(A)$ and vice versa. So,
\begin{eqnarray}
0=<f_1(A)(\Ker f_1(A),\Ker f_2(A))>=<\Ker f_1(A),f_1(A)(\Ker f_2(A)>=\nonumber\\
<\Ker f_1(A),\Ker f_2(A)>. \nonumber
\end{eqnarray}

(3) If $v\in\Ker A$, then $<v,Aw>=<Av,w>=0.$ 
\end{proof}
\begin{proc}\label{4-s-a}
Let $V$ be a symplectic vector space of dimension $4$ and $A$ be a non scalar, symplectic 
self-adjoint operator on $V$. Then the minimal polynomial $f(t)$ of $A$ is of second order
and
\begin{enumerate}
\item  if $f(t)$ has complex roots, then $A$ supplies $V$ with a complex structure, all its proper
subspaces are complex lines and these lines understood as 2-dimensional real planes are 
Lagrangian \emph{(elliptic case);}
\item if $f(t)$ has different real roots $\lambda_1, \lambda_2$, then eigenspaces 
$\Ker(A-\lambda_i\id_V)$ are symplectic orthogonal and non-Lagrangian planes 
\emph{(hyperbolic case);}
\item if roots of $f(t)$ coincide, i.e., $f(t)=(t-\lambda)^2$, then the unique eigenspace 
$W=\Ker (A-\lambda\id_V)$ of $A$ is a Lagrangian plane, which coincides with 
$\im(A-\lambda\id_V)$. Lagrangian planes intersecting $W$ by a line are cyclic
subspaces of $A$ of dimension 2 and vice versa \emph{(parabolic case).}
\end{enumerate}
\end{proc}
\begin{proof}
First, note that a linear operator possesses a cyclic subspace whose dimension equals to the
degree of its minimal polynomial. By lemma \ref{elementary}, (1), the symplectic form vanishes 
on a cyclic subspace and hence its dimension cannot be greater than 2. It cannot be 1-dimensional, since $A$ is not scalar. Hence the minimal polynomial of $A$ is of second order.

(1) If $\lambda_1=a+b\mathbf{i}$ and $B=b^{-1}(A-a\id_V)$, then $B^2=-\id_V$. So, $B$
supplies $V$ with a structure of $\C$--vector space and $A$ is $\C$-linear. In view of lemma
\ref{elementary}, (1), other assertions directly follows from this fact.

(2)In this case $V$ is the direct sum of eigenspaces.  By lemma \ref{elementary}, (2), they are
symplectic orthogonal, and, so, none of them could be of dimension 1. Indeed, the symplectic
orthogonal to a line is a 3-dimensional subspace containing this line. So, the eigenspaces are 
2-dimensional, and none of them can be Lagrangian, since they are symplectic orthogonal. 

(3) If $B=A-\lambda\id_V$, then $B^2=0$, i.e., $\im B\subset \Ker B, W=\Ker B$ and $B$ is 
symplectic self-adjoint. If
$\im B$ is 1-dimensional, then $\dim(\Ker B)=\dim W=3$. On the other hand, by  lemma \ref{elementary}, (3), $W$ is symplectic orthogonal to $\im B$ and, therefore, coincides with
the symplectic orthogonal complement of $\im B$. If $0\neq v\in \im B$ and $v=Bu$, then
$u\notin W$ and hence is not symplectic orthogonal to $v$. But this contradicts to
the fact that $<u,Bu>=0$ (lemma \ref{elementary}, (1)). Hence $\dim W=2, W=\im B$.
Moreover, $W$ is Lagrangian ss a symplectic orthogonal to $\im B$ subspace. 

Finally, let $L$ be a Lagrangan plane that intersects $W$ by a line $\ell$ and $u\in L\setminus \ell$.
Then $0\neq Bu\in W$. But the span $L'$ of $u$ and $Bu$ is a Lagrangian plane. It is symplectic orthogonal to $\ell$, since such are $u$ and $Bu$. But any Lagrangian plane, which is symplectic orthogonal to a line, contains this line. In particular, $L'$ contains $\ell$ and, therefore,   
$L'\cap W=\ell$. So, $Bu\in\ell$ and hence $L=L'$ and is cyclic. The converse is obvious.
\end{proof}
\begin{rmk}
It is curious to observe that a Lagrangian plane $W$ in a symplectic $V^4$ defines a symplectic 
self-adjoint operator $B$ such that $B^2=0$ and $\Ker B=W$. Such an operator is unique up to 
a scalar factor. Indeed, if $u\in V\setminus W$, then $Bu$ should belong to $W$ and be 
symplectic orthogonal to $u$. Since the symplectic orthogonal complement of $u$ intersects
$W$ by a line, say, $\ell_u$, $Bu$ should belong to $\ell_u$ and hence is unique up to a factor.
In order to construct one such operator choose a non-Lagrangian and complementary to $W$ 
plane $U$ and an isomorphism $h:U\rw W$ such that $h(u)\in \ell_u$. Then any $v\in V$ is
uniquely presented in the form $v=u+w, \,u\in U, \,w\in W,$ and we put $Bv:=h(u)$. 
\end{rmk}
Recall that a 2-dimensional unital associative algebra is isomorphic to the algebra of 
\emph{$\zeta$-complex} numbers whose elements are of the form
$x+y\zeta$ with $\zeta^2=-1, 0$ or $1$ and called \emph{double, dual} or \emph{complex},
numbers, respectively. Accordingly, denote these algebras by $\C_-(=\C), \C_0$ and $\C_+$.
\begin{cor}
There are 3 isomorphism classes of basic algebras for $n=2$ represented by algebras 
$\C_{\pm}$ and $\C_0$.
\end{cor}
\begin{proof}
Any basic algebra for $n=2$ is generated by $\id_V$ and a non-scalar symplectic self-adjoint 
operator $A$. A linear combination of these operators is an operator $B$ such that 
$B^2=\pm\id_V \mbox{or} \;0$. It follows from proposition \ref{4-s-a} that two such operators of 
the same type are symplectically isomorphic.
\end{proof}
\begin{rmk}
There is a simple approach to description of symplectic self-adjoint operators based on following
observations.
\begin{enumerate} 
\item If $W$ is an $n$--dimensional $\R$--vector space, then $V=W\oplus W^*$ is symplectic
with respect to the form $\lfloor(w_1,\varphi_1),(w_2,\varphi_2)\rceil=\varphi_1(w_2)-\varphi_2(w_1)$.
\item Let $W$ and $W'$ be Lagrangian subspaces of a symplectic vector space $V$, which are 
complementary to each other. Then $\imath: w'\mapsto <w',\cdot>, \,w'\in W'$, is an isomorphism between $W'$ and $W^*$, which extends to the symplectic isomorphism $v=w+w'\mapsto w\oplus\imath(w')$ between $V$ and $W\oplus W^*$.
\item If $F\in \End W$, then $F\oplus F^*\in \End W\oplus W^*$ is symplectic self-adjoint.
\end{enumerate}
Now it is easily follows from (1)-(3) and \emph{lemma \ref{elementary}, (1)}, that any symplectic 
self-adjoint operator is of the form $F\oplus F^*$. 

The approach we have chosen above is motivated by the fact that it is better adapted to the
solution singularities theory.
\end{rmk}


\subsection{Geometric singularities of solutions of PDEs}\label{sng}
Here we shall assemble some facts concerning classification of singularities of R-manifolds that
are necessary to encode the meaning of classical MA-equations (see \cite{Vsing}, \cite{Vin}). The classification 
group is that of symmetries of the Cartan distribution $\cC_k(E,n)$. According to the 
Lie-B\"acklund theorem this group consists of diffeomorphisms of $E$ lifted to $J^k(E,n)$ if
$m>1$ and of lifted contact transformations of $J^1(E,n)$ if $m=1$.
 
The simplest question here is the first order classification, i.e., classification of tangent spaces
to R-manifolds at singular points. The first result in this direction is that this classifications is 
equivalent to classification of bends (see subsection \ref{jets}). The second important fact is 
that the classification of bends depends only on the dimension, say, $s$, of bends, i.e., does 
not depends on $m\geq1, k>1$ and $n\geq s$. This reduces the problem to the case 
$m=1, k=2, n=s$.

As it follows from the definition, the bend at $\theta$ of an R--manifold $U$ is a subspace of
 $T_{\theta}\Phi$ where $\Phi$ is the fiber of $\pi_{k,k-1}$ passing through $\theta$. If $m=1$,
 then  $T_{\theta}\Phi$ is identified with the space $P_{k,n}$ of homogeneous polynomials of 
 degree $k$ in $n$ variables over $\R$ (see \cite{KLV}). In these terms, the $s$--bend of an 
 R--manifold passing through $\theta$ can be characterized as a special subspace $\gimel$ 
 in  $P_{k,s}$, which, by abusing the language, we shall also call a bend. 
 
 It holds the following key assertion (see \cite{Vin})
 \begin{proc}\label{bend}
 If $\gimel\subset P_{k,s}$ is an $s$--bend, then the span of polynomials in $P_{k+1,s}$ whose
 derivatives belong to $\gimel$ is also an $s$--bend in $P_{k+1,s}$.
 \end{proc}
 
In the case of MA--equations we have $m=1, k=n=2$. Peculiarity of this case is twofold. First, 
it is easy to see that any 2-dimensional subspace of $P_{2,2}$ in the condition of proposition \ref{bend}, and hence any 2-dimensional subspace tangent to a fiber of $\pi_{2,1}$ is a 2-bend. 
Then, a second order PDE $\cE\subset J^2(E^3,2)$, generally, intersects fibers of $\pi_{2,1}$ 
by 2-dimensional submanifolds. So, tangent spaces to these submanifolds are 2-bends.


\subsection{2-bends, basic algebras and $\zeta$--holomorphic functions}
Proposition \ref{bend} establishes a natural relation between bends and basic algebras, which 
we shall make explicit in the case $n=2, m=1$ (see \cite{Vin} for the general case). 

Let $\gimel\subset P_{k,2}, \,k\geq 2,$ be a bend. 
Then, according to proposition \ref{bend}, there is a polynomial $f(x,y)\in P_{k+1,2}$ whose
derivatives $f_x$ and $f_y$ form a basis of $\gimel$. By the same proposition there should 
be a non proportional to $f$ polynomial $g$ whose derivatives belong to $\gimel$. Hence
$g_x=\alpha f_x+\beta f_y, \,g_y=\gamma f_x+\delta f_y, \,\alpha,\dots,\delta\in\R$, and
$(\alpha f_x+\beta f_y)_y=(\gamma f_x+\delta f_y)_x$, or, equivalently, 
\begin{equation}\label{lapl}
\gamma f_{xx}+(\delta-\alpha)f_{xy}+\beta f_{yy}=0.
\end{equation}
Since $g$ is not proportional to $f$, coefficients of equation \ref{lapl} do not vanish simultaneously,
and this equation can be brought to one of the forms $u_{xx}\pm u_{yy}=0$ or $u_{xx}=0$.
Accordingly, one of basic basic algebras $\R_{\pm}$ or $\R_0$ is associated with $\gimel$. 
More exactly, this basic algebra is generated by the operator whose matrix is
$$
\left(\begin{array}{cc}
\alpha & \beta \\ \gamma & \delta
\end{array}\right).
$$
Normal forms of 2-bends are easily deduced from this fact and they are 
$$
\mathrm{Span}(\mathrm{Re}\,z^k,\mathrm{Im}\,z^k), \quad z=x+\zeta y,
$$ 
where $\mathrm{Re}$ and $\mathrm{Im}$ refer to the ``real" and ``imaginary"
parts of a $\zeta$-complex number (see subsection \ref{JS}), respectively. 

Examples of R-manifolds with one singular point, in which these normal 2--bends are realized, 
are easily constructed in terms of $\zeta$-holomorphic functions. Namely, let $N$ be a 
2-dimensional manifold. An $C^{\infty}(N)$-homomorphism $A\colon D(N)\rw D(N)$  such 
that $A^2=\zeta^2\id_{D(N)}$ supplies $N$ with a structure of a \emph{$\zeta$-complex curve}. 
A function $u\in C^{\infty}(N)$ is \emph{$\zeta$-harmonic} if $A^*(du)=dv$ for a function 
$v\in C^{\infty}(N)$. Here $A^*\colon \Lambda^1(N)\rw \Lambda^1(N)$ is the dual to $A$ homomorphism. The function $v$ is unique up to a constant and is called ($\zeta$-)conjugate 
to $u$. It is also $\zeta$-harmonic and together with $u$ form a $\C_{\zeta}$-valued function 
$f=u+\zeta v$. Such a function is called \emph{$\zeta$-holomorphic}.

Generally, non-constant $\zeta$--holomorphic functions on $N$ exist only locally. A pair of 
$\zeta$-conjugate local functions $x$ and $y$ on $N$ forms a local $\zeta$-complex
chart on $N$, and, locally, $\zeta$-holomorphic functions may be viewed as functions of
$\zeta$-complex variable $z=x+\zeta y$. In particular, the transition function between two 
$\zeta$-complex charts is $\zeta$-holomorphic. Standard complex curves are, obviously, 
$\zeta$-complex ones for $\zeta^2=-1$. If $\zeta^2=1$ (resp., $\zeta^2=0$), then, as it is 
easy to see, a $\zeta$-complex curve is a 2-dimensional manifold supplied with two transversal
to each other 1-dimensional distributions (resp., one 1-dimensional distribution).

Locally, the condition $A^*(du)=dv$ is equivalent to $(d\circ A\circ d)(u)=0$. The operator 
$\Delta_{\zeta}=d\circ A\circ d$ will be called the \emph{$\zeta$-Laplace operator}, since 
in a $\zeta$-complex chart it reads $\Delta_{\zeta}=\p^2/\p x^2-\zeta^2\p^2/\p y^2$. Accordingly,
$u_{xx}-\zeta^2\ u_{yy}=0$ is the \emph{$\zeta$--Laplace equation}. So, in these terms, $\zeta$-harmonic functions are solutions of the $\zeta$-Laplace equation. This equation express the
compatibility condition for the \emph{$\zeta$-Cauchy-Riemann equations} $A^*(du)=dv$, 
which in a $\zeta$-complex chart read  $u_x=-\zeta^2 v_y, \;u_y=\zeta^2 v_x$. In other 
words, a function $f=u+\zeta v$ is  $\zeta$-holomorphic iff its \emph{real} and \emph{imaginary} 
parts $u=\mathrm{Re}\,f$ and $v=\mathrm{Im}\,f$, respectively, satisfy the $\zeta$-Cauchy-Riemann equations.

Now consider the $(k-2)$-th prolongation $\cE_{(k-2)}\subset J^k(E,n)$ of the $\zeta$-Laplace equation $\cE\subset J^{2}(E,n)$. It is given by equations
\begin{equation}\label{norm}
u_{2+r,s}-\zeta^2u_{r,s+2}=0, \quad r+s\leq k-2 
\end{equation}
with $u_{p,q}$ being the function on $J^{p+q}(E,n)$ representing the operator 
$\p^{p+q}/\p x^p\p y^q$ in a standard local jet-chart. Intersections of $\cE_{(k-2)}$ with fibers 
of $\pi_{k,k-1}$ are given by equations 
$$
u_{p,q}=\mathrm{const}, \;p+q\leq k-1, \quad u_{2+r,s}-\zeta^2u_{r,s+2}=0, \quad r+s=k-2 
$$
and hence  are 2-dimensional. Vectors
$$
\nu_1=\sum_{r=0}^{[k/2]}\zeta^{2r}\p/\p u_{2r,k-2r}, \quad
\nu_2=\sum_{r=0}^{[k-1/2]}\zeta^{2r}\p/\p u_{1+2r,k-1-2r} 
$$
form a basis of the tangent space to such an intersection. In coordinates the identification of 
tangent to fibers of $\pi_{k,k-1}$ vectors with homogeneous polynomials of degree $k$ (see 
subsecton \ref{sng}) looks as 
$$
\frac{\p}{\p u_{r,k-r}}\quad\longleftrightarrow \quad\frac{1}{r!(k-r!}x^ry^{k-r}
$$
So, the vectors $\nu_1$ and $\nu_2$ are identified with $\frac{1}{k!}\mathrm{Re}(x+\zeta y)^k$ 
and $\frac{1}{k!}\mathrm{Im}(x+\zeta y)^k$, respectively.
 
 Let now $1<l\in \Z$ and consider the R-manifold $L_{k,l}^{\zeta}$ in $J^k(E,n)$
 given by equations (\ref{norm}) and equations:

 \begin{eqnarray}\label{Rl}
x=\frac{1}{(k+\frac{1}{l})!}\mathrm{Re}(u_{(k,0)}+\zeta u_{(k-1,1)})^k, \quad
y=\frac{1}{(k+\frac{1}{l})!}\mathrm{Im}(u_{(k,0)}+\zeta u_{(k-1,1)})^k, \nonumber \\
u_{(k-r,0)}=\frac{1}{(r+\frac{1}{l})![(k+\frac{1}{l})!]^{{lr}}}\mathrm{Re}(u_{k,0}+\zeta u_{k-1,1})^{lr+1},
\quad  1\leq r\leq k, \\
u_{(k-r-1,1)}=\frac{1}{(r+\frac{1}{l})![(k+\frac{1}{l})!]^{{lr}}}\mathrm{Im}(u_{k,0}+\zeta u_{k-1,1})^{lr+1},
\quad  1\leq r\leq k. \nonumber
\end{eqnarray}
where $(s+\frac{1}{l})!:=(1+\frac{1}{l})(2+\frac{1}{l})\dots(s+\frac{1}{l}), \,s\in \Z_+$. This R-manifold
has the unique singular point $u_{(r,s)}=0,  \,r+s\leq k$, in which the bend is
$\mathrm{Span}(\mathrm{Re}\,z^k,\mathrm{Im}\,z^k)$. It may be viewed as the real part of the
$k$--th jet of the multivalued $\zeta$-holomorfic function $f(z)=z^{k+\frac{1}{l}}$. 

It is not difficult to show that R-manifolds $L_{k,l}^{\zeta}$ corresponding to different $l$'s are
not equivalent, while they have the common bend.
\begin{rmk}
The problem of finding $s$-bend solutions for relevant equations of mathematical physics and differential geometry was not yet systematically studied even for $s=2$. Some  2-bend solutions 
of the vacuum Einstein equations were constructed in \cite{SVV}. Among them are 
\emph{foam-like} solutions describing ``parallel universes" separated by singularities. The 
\emph{square root} of the Schwarzschild metric is the simplest of that kind.
\end{rmk}


\subsection{Intrinsic definition of classical Monge-Amp\`ere equations}
Now we are ready to formulate a \emph{conceptual definition} of classical MA-equations. We say 
that a Legendrian submanifold $L\subset M$ is invariant with respect to a JS-subalgebra $\cA$ in
$\mathrm{Sad}(\vk)$ if $T_xL$ is $\cA_x$-invariant for all $x\in M$.
\begin{defi}\label{concept}
A classical \emph{Monge-Amp\`ere problem} is that of finding Legendrian submanifolds 
of a 5-dimensional contact manifold $M$ that are invariant with respect to a given basic
algebra $\cA\subset\mathrm{Sad}(\vk)$.  $\cA$--invariant  Legendrian submanifolds 
are called \emph{solutions} of this problem.
\end{defi}

In order to connect definitions \ref{MA} and \ref{concept}, recall that according to the classical
Darboux lemma, a given contact form $\omega$ on a $2n+1$--dimensional manifold locally
admits a \emph{Darboux chart} $(x_i,p_i,u), \,i=1,\dots,n,$, in which it takes the \emph{canonical}
form $\omega=du-\sum_{i=1}^{n}p_idx_i$. A regular with respect to this chart Legendrian 
submanifold is given by equations
$$
u=f(x), \;p_i=\frac{\p f(x)}{\p x_i}, \,\;i=1,\dots,n, \;\mbox{with} \; x=(x_1,\dots,x_n).
$$
To underline that such a submanifold is described by a function $f(x)$ we denote it by $L_f$. 
\begin{proc}
Solutions of equation \emph{(\ref{MA})} are solutions of a classical Monge-Amp\`ere problem
for Legendrian submanifolds of the form $L_f$ and vice versa.
\end{proc}
\begin{proof}
A natural basis of $\cD$ in the Darboux  chart $(x_i,p_i,u)$ is 
$$
\p_{x_1}+p_1\p u, \quad\p_{x_2}+p_2\p u, \quad\p_{p_1}, \quad\p_{p_2},
$$
and the operator $\mathfrak{A}$ whose matrix in this basis is
$$
\left(\begin{array}{cccc}
B & -2A & 0 & -2N \\ 2C & -B & 2N & 0 \\ 0 & 2D & B & 2C \\ -2D & 0 & -2A & -B 
\end{array}\right)
$$
is such that $\mathfrak{A}^2= \Delta I$ where $I$ is the unit matrix and $\Delta=B^2-4AC+4ND$. Recall that equation 
(\ref{MA}) is elliptic (resp., parabolic, or hyperbolic) if $\Delta<0$ (resp., $=0$, or $>0$). 
$\mathfrak{A}$ and $\id_{\cD}$ span a basic algebra $\cA$. Obviously, solutions of the 
corresponding Monge-Amp\`ere problem are $\mathfrak{A}$--invariant Legendrian 
submanifolds.

Vector fields 
$$
Z_1=\p_{x_1}+p_1\p u+f_{x_1x_1}\p_{p_1}+f_{x_1x_2}\p_{p_2}, \quad 
Z_2=\p_{x_2}+p_2\p u+f_{x_1x_2}\p_{p_1}+f_{x_2x_2}\p_{p_2}
$$ 
are tangent to $L_f$, while
$$
\mathfrak{A}(Z_1)=(B-2f_{x_1x_2}N)Z_1+2(C+f_{x_1x_1}N)Z_2-2E\p p_2
$$
where
$E=N(f_{x_1x_1}f_{x_2x_2}-f_{x_1x_2}^2)+Af_{x_1x_1}+Bf_{x_1x_2}+Cf_{x_2x_2}+D$,
and similarly for $\mathfrak{A}(Z_2)$. This shows that $L_f$ is $\mathfrak{A}$--invariant
iff $E=0$.
\end{proof}

Thus this proposition establishes a one-to-one correspondence between MA-equations and
2-dimensional basic algebras on contact 5-folds. This correspondence put in evidence the
fact that any solution of an MA-equation $\cE$ is a bi-dimensional manifold $L$ 
supplied with an algebra of endomorphisms of   $D(L)$. Namely, this algebra, denoted by 
$\cA_L$, is composed from restricted to $L$ elements of the associated with $\cE$ basic 
algebra $\cA$.  If $\cE$ is elliptic (resp., parabolic or hyperbolic), then $\cA_L$ is of type 
$\C_-$ (resp., $\C_ 0$ or $\C_+$), i.e., $L$ is a $\zeta$-complex curve for the 
corresponding $\zeta$. In fact, solutions of arbitrary second order PDE in two 
independent variables are naturally supplied with a such an algebra of endomorphisms.
This is due to the fact that any 2-dimensional subspace tangent to a fiber of $\pi_{2,1}$
is a 2-bend.

\begin{rmk}
One of central questions in solution singularity theory is to what extent a given PDE $\cE$ 
is predetermined by behavior of singularities of its solutions. More exactly, a series of subsidiary
equations $\cE_{\Sigma}$ where $\Sigma$ is a singularity type admitted by solutions of $\cE$
is associated with $\cE$. See \cite{LMSV} for some examples. The \emph{reconstruction problem} 
is : whether $\cE$ can be reconstructed if all equations $\cE_{\Sigma}$ are known? A spectacular,
though implicit, historical example of solution of this problem is Maxwell's deduction of his
famous equations from previously found elementary (Coulomb, ... , Faraday) laws of electricity 
and magnetism, which, according to the modern point of view, describe behavior of singularities 
of electromagnetic fields. In this context classical MA-equations are distinguished by the fact 
that the reconstruction problem for them  is a tautology, namely, 
by definition \emph{\ref{concept}}.
\end{rmk}

\subsection{Concluding remarks}
The above interpretation of classical MA-equations was enlightening in the development 
fundamentals of solution singularities theory as one of key examples. Indeed, the discussed
in this section constructions and results can be generalized to dimensions $n>2$. For instance,
definition \ref{concept}  in an obvious manner generalizes to higher dimensions. It is not difficult 
to see that $n$--dimensional Monge-Amp\`ere problem is described by a system of 
$\frac{n(n-1)}{2}$ second order PDEs, which which may be viewed as another natural 
generalization of classical MA-equations.

This interpretation is useful in the study of classical MA-equations themselves. In particular, it
suggests more exact techniques for integrating concrete MA-equations, finding their classical symmetries, conservation laws, etc. For instance, classical infinitesimal symmetries of the 
MA-equation associated with a basic algebra $\cA$ may be defined as contact fields $Z$ such
that $[Z,\cA]\subset \cA$ and this interpretation directly leads to a simple computational procedure.

Another example we would like to
mention here is the classical problem of contact classification of MA-equations, which, 
essentially, is equivalent to the problem of finding a sufficient number of their scalar differential invariants. Most systematically this problem was studied   by V.~Lychagin and his collaborators (V.~Rubtsov, B.~Kruglikov,  A.~Kushner and others) in last three decades. These authors exploited Lychagin's idea to represent MA-equations in terms of \emph{effective} 2-forms. Most complete 
results in this approach were obtained by A.~Kushner (see short note \cite{K}). Full details of his
work are reproduced in monograph \cite{KLR} together with earlier results of these authors.
In these works scalar differential invariants of hyperbolic and elliptic MA-equations are constructed
indirectly as differential invariants of the associated $e$--structures. On the contrary, definition
\ref{concept} allows a direct construction of scalar differential invariants in terms of operators of 
the corresponding basic algebra, which leads to more complete and exact results (see \cite{MVY}, \cite{DV}, \cite{VD}, \cite{Diego}).

\end{document}